\documentclass{article}
\usepackage{amsthm,amsmath,amssymb,authblk}
\newtheorem{Theorem}{Theorem}
\newtheorem{Lemma}[Theorem]{Lemma}
\newtheorem{Proposition}[Theorem]{Proposition}
\newtheorem{Conjecture}[Theorem]{Conjecture}
\theoremstyle{definition}

\newtheorem{Remark}[Theorem]{Remark}
\newcommand{\bv}{\mathbf{v}}
\newcommand{\ZZZ}{\mathbb Z}

\newcommand{\abs}[1]{\lvert#1\rvert}
\DeclareMathOperator{\Gal}{Gal}

\begin{document}
\title{Galois groups of multivariate Tutte polynomials}
\author{Adam Bohn}
\author[1]{Peter J. Cameron}
\author[2]{Peter M\"uller}
\affil[1]{School of Mathematical Sciences, Queen Mary, University of London, Mile End Road, London E1 4NS, UK.}
\affil[2]{Institut f\"ur Mathematik, Universit\"at W\"urzburg, Campus Hubland Nord, 97074 W\"urzburg, Germany}
\maketitle

\begin{abstract}
The multivariate Tutte polynomial $\hat Z_M$ of a matroid $M$ is a generalization of the standard two-variable version, obtained by assigning a separate variable $v_e$ to each element $e$ of the ground set $E$.  It encodes the full structure of $M$.  Let $\bv = \{v_e\}_{e\in E}$, let $K$ be an arbitrary field, and suppose $M$ is connected.  We show that $\hat Z_M$ is irreducible over $K(\bv)$, and give three self-contained proofs that the Galois group of $\hat Z_M$ over $K(\bv)$ is the symmetric group of degree $n$, where $n$ is the rank of $M$.  An immediate consequence of this result is that the Galois group of the multivariate Tutte polynomial of any matroid is a direct product of symmetric groups.  Finally, we conjecture a similar result for the standard Tutte polynomial of a connected matroid.
\end{abstract}

Let $M$ be a finite matroid on the set $E$. The rank of $M$ is denoted by $r(M)$, and $r_M$ is the rank function on $M$. With this notation we have $r(M)=r_M(E)$.  To avoid degenerate examples and exceptions a connected matroid will be assumed throughout to have positive rank (our results are trivial for a matroid having zero rank). Following the usual notation in matroid theory, we will write $E\setminus e$ instead of $E\setminus\{e\}$ for $e\in E$, and denote by $M|A$ the restriction of $M$ to some $A\subset E$.

For each $e\in E$ let $v_e$ be a variable, and let $\bv$ be the collection of these variables.  If $A$ is a subset of $E$, we will denote by $\bv_A$ the set $\{v_e\}_{e\in A}$.  In \cite{Sokal:Tutte} Sokal defines the following multivariate version of the Tutte polynomial of a matroid $M$ \footnote{The multivariate Tutte polynomial for matroids has in fact been discovered a number of times; it appears, for example, in \cite{Kung:Twelve} as the ``Tugger polynomial''.}.

For another variable $q$ set
\[
\tilde Z_M(q,\bv)=\sum_{A\subseteq E}q^{-r_M(A)}\prod_{e\in A}v_e.
\]
Then $\tilde Z_M(q,\bv)$ is a polynomial in $\frac{1}{q}$ with coefficients in $\ZZZ[\bv]$.

For our purpose it is more convenient to use the following minor modification:
\[
\hat Z_M(q,\bv)=\sum_{A\subseteq E}q^{r(M)-r_M(A)}\prod_{e\in A}v_e.
\]
Then
\[
\hat Z_M(q,\bv)=q^{r(M)}\tilde Z_M(q,\bv),
\]
and $\hat Z_M(q,\bv)$ is a polynomial of degree $r(M)$ in $q$, which is monic if $M$ contains no loops.  In particular, if $M$ is connected then $\hat Z_M(q,\bv)$ is monic.  Combinatorially, $\hat Z_M(q,\bv)$ is a generating function for the content and rank of the subsets of $E$, and thus encodes all of the information about $M$.

By making the substitutions
\begin{align*}
    &q\leftarrow (x-1)(y-1)\\
    &v_e\leftarrow y-1
\end{align*}
for each $e\in E$, and multiplying by a prefactor $(y-1)^{-r(M)}$, we obtain the standard bivariate Tutte polynomial:
    \[T_M(x,y)=\sum_{A\subseteq E}(x-1)^{r(M)-r_M(A)}(y-1)^{|A|-r_M(A)}.\]

Thus $T_M$ is essentially equivalent to a special case of $\hat Z_M$ in which the same variable is assigned to every element of $E$.

\begin{Theorem}\label{T:main}
Let $M$ be a finite connected matroid with positive rank $n=r(M)$, and let $\hat Z_M(q,\bv)$ be as defined above. Let $K$ be an arbitrary field.  Then the Galois group of $\hat Z_M(q,\bv)$ over $K(\bv)$ is the symmetric group on the $n$ roots of $\hat Z_M(q,\bv)$.
\end{Theorem}

For $e\in E$, let $M\setminus e$ be the deletion of $e$, and $M/e$ the contraction of $e$. Note that $M\setminus e$ and $M/e$ are matroids on the set $E\setminus e$.  The essential tool for our first proof is a theorem of Tutte (see \cite[Theorem 4.3.1]{Oxley:Book}), which says that connectivity of $M$ implies that at least one of the matroids $M\setminus e$ or $M/e$ is connected.  Since $M$ is connected, $e$ is not a coloop, so $r(M\setminus e)=r_M(E\setminus e)=r_M(E)=r(M)$. By \cite[Prop.\ 3.1.6]{Oxley:Book} we have that $r(M/e)=r_M(E)-r_M(e)$. Now $r_M(e)=1$, since $e$ is not a loop. So $r(M/e)=r(M)-1$.

The proof will be based on some lemmas.

\begin{Lemma}\label{L:sum}
Let $M$ be a finite connected matroid and $e\in E$. Then
\[
\hat Z_M=\hat Z_{M\setminus e}+v_e\hat Z_{M/e}.
\]
\end{Lemma}

\begin{proof}
Since $M$ is connected, $e$ is neither a loop nor a coloop. By \cite[(4.18a)]{Sokal:Tutte} $\tilde Z_M=\tilde Z_{M\setminus
  e}+\frac{v_e}{q}\tilde Z_{M/e}$, hence
\[
\hat Z_M=q^{r(M)-r(M\setminus e)}\hat Z_{M\setminus
  e}+q^{r(M)-r(M/e)}\frac{v_e}{q}\hat Z_{M/e}.
\]
The claim then follows from the previous determination of the ranks of $E\setminus e$ and $E/e$.
\end{proof}

As an intermediate step in the proof of the theorem, we need to know that $\hat Z_M$ is irreducible over $K(\bv)$. As $T_M$ is essentially a specialization of $\hat Z_M$, this would follow from \cite{Merino:Tutte} in the case where $K$ has characteristic zero. However, the multivariate case allows for a much simpler proof, and one which holds for any characteristic.

\begin{Lemma}\label{L:irr}
Let $M$ be a finite connected matroid. Then $\hat Z_M$ is irreducible over $K(\bv)$.
\end{Lemma}

\begin{proof}
The induction proof is most conveniently formulated by considering a counterexample $M$ where $r(M)$ is minimal; among those counterexamples, we pick one where $\abs{E}$ is minimal. Clearly, the result holds for $r(M)=1$, so $r(M)\ge2$. Pick $e\in E$. By Lemma \ref{L:sum}, $\hat Z_M=\hat Z_{M\setminus e}+v_e\hat Z_{M/e}$. Note that $v_e$ does not appear in $\hat Z_{M\setminus e}$ and $\hat Z_{M/e}$. If $M\setminus e$ is connected, then $\hat Z_{M\setminus e}$ is irreducible by minimality of $\abs{E}$. As $\hat Z_M$ and $\hat Z_{M\setminus e}$ have the same degree, setting $v_e=0$ shows that $\hat Z_M$ is irreducible, a contradiction.  So $M\setminus e$ is not connected, which by Tutte's theorem means that $M/e$ is connected.  So $r(M/e)\ge1$ (because $r(M)\ge2$), and $\hat Z_{M/e}$ is monic.  Note also that because $M$ is loopless, so too is $M\setminus e$, and hence $\hat Z_{M\setminus e}$ is also monic.

Now, consider a non-trivial factorization of $\hat Z_M$. Since $\hat Z_M$ is monic and linear in $v_e$, we can write $\hat Z_M=(U+v_eV)W$, where $U,V,W$ are polynomials in $K[\bv][q]$ in which $v_e$ does not appear, and where each factor has positive degree in $q$.

So $(U+v_eV)W=\hat Z_{M\setminus e}+v_e\hat Z_{M/e}$. Comparing coefficients with respect to $v_e$ gives $UW=\hat Z_{M\setminus e}$ and $VW=\hat Z_{M/e}$. By minimality of the counterexample, $\hat Z_{M/e}$ is irreducible. But $W$ has positive degree in $q$, so $V=1$ and $W=\hat Z_{M/e}$. Thus $U\hat Z_{M/e}=\hat Z_{M\setminus e}$. Now, $\hat Z_{M/e}$ and $\hat Z_{M\setminus e}$ are monic of degrees $r(M)-1$ and $r(M)$ respectively. So $U=q+\beta$ for some $\beta\in K[\bv]$.  Let $\bar \bv=\bv\setminus \{v_e\}$, and note that
\[
\hat Z_{M\setminus e}(1,\bar \bv)=\prod_{i\in E\setminus e}(1+v_i)=\hat Z_{M/e}(1,\bar \bv),
\]
so $\beta=0$. Now setting $q=0$ gives $\hat Z_{M\setminus e}(0,\bar \bv)=0$. This means that there are no bases in $M\setminus e$, which is only possible if every element of $E\setminus e$ is a loop.  So we have a contradiction.
\end{proof}

In order to prove the theorem, we need more precise information about how Galois groups behave under specializations of parameters. The next result is well-known, it follows for instance from \cite[Theorem IX.2.9]{Lang:Alg}.

\begin{Proposition}\label{P:gal}
Let $R$ be an integral domain which is integrally closed in its quotient field $F$. Let $f\in R[X]$ be monic and irreducible over $F$. Let $R\to k$, $r\mapsto\bar r$ be a homomorphism to a field $k$. If $\bar f\in k[X]$ is separable, then $\Gal(\bar f/k)$ is a subgroup of $\Gal(f/F)$.
\end{Proposition}

The following two lemmas can be obtained through applications of this proposition.

\begin{Lemma}\label{L:restrict}
Let $A$ be a subset of $E$. Then $\Gal(\hat Z_{M|A}/K(\bv_A))$ is a subgroup of $\Gal(\hat Z_M/K(\bv))$.
\end{Lemma}

\begin{proof}
Let $B$ be such that $A\subset B\subseteq E$, and let $e$ be an element of $B\setminus A$.  Note that removing $e$ from $B$ corresponds to specializing $v_e$ to zero in $\hat Z_{M|B}$.  Let $R=K(\bv_{B\setminus e})[v_e]$, and let $I$ be the maximal ideal of $R$ generated by $v_e$.  The image of $\hat Z_M$ in the canonical homomorphism $R\to R/I$ is either $q\hat Z_{M|(B\setminus e)}$ or $\hat Z_{M|(B\setminus e)}$, depending on whether or not $e$ is a coloop.  In both cases we have a separable polynomial, as the presence of a repeated irreducible factor would contradict the fact that $\hat Z_{M|(B\setminus e)}$ is linear in the elements of $\bv_{B\setminus e}$.  Furthermore, $R$ is integrally closed in its quotient field $K(\bv)$. So we have that $\Gal(\hat Z_{M|(B\setminus e)}/K(\bv_{B\setminus e}))\leq\Gal(\hat Z_{M|B}/K(\bv_B))$ by Proposition \ref{P:gal}, and the result follows by induction.
\end{proof}

\begin{Lemma}\label{L:U+yV}
Let $y$ be a variable over the field $k$, and $U,V\in k[X]$ with $\deg V=n-1$, and $U$ monic of degree $n$ (where $n\ge2$). Suppose that $f(X)=U(X)+yV(X)$ is irreducible over $k(y)$ (which is equivalent to $U$ and $V$ being relatively prime). If $\Gal(U/k)=S_n$ or $\Gal(V/k)=S_{n-1}$, then $\Gal(f/k(y))=S_n$.
\end{Lemma}

\begin{proof}
First suppose that $\Gal(U/k)=S_n$. Then the assertion follows immediately from Proposition \ref{P:gal} by setting $R=k[y]$ and considering the homomorphism $R\to k$, $h(y)\mapsto h(0)$.

Now assume that $\Gal(V/k)=S_{n-1}$. Set $t=1/y$ and replace $f(X)=U(X)+yV(X)=U(X)+\frac{1}{t}V(X)$ with $t$ times the reciprocal of $f(X)$, that is set $\hat f(X)=X^n(tU(1/X)+V(1/X))$. Clearly, $k(t)=k(y)$ and $\Gal(f/k(y))=\Gal(\hat f/k(t))$. The coefficient of $X^n$ in $\hat f$ is $tu+v$, where $u$ and $v$ are the constant terms of $U$ and $V$. If $v=0$, then $V$ has the root $0$. However, $V$ is irreducible since $\Gal(V/k)=S_{n-1}$. So $n=2$. The result clearly holds in this case because $f$ is then irreducible of degree $2$.

So assume $v\ne0$. Let $R\subset k(t)$ be the localization of $k[t]$ with respect to the ideal $(t)$, so $R$ consists of the fractions $p(t)/q(t)$ with $q(0)\ne0$. Note that $\frac{1}{tu+v}\hat f$ is monic with coefficients in $R$. Also, $R$ (as a local ring) is integrally closed in $k(t)$. Let $R\to k$ be the homomorphism given by $p(t)/q(t)\mapsto p(0)/q(0)$. Proposition \ref{P:gal} then gives $\Gal(\hat f/k(t))\ge\Gal(X^nV(1/X)/k)=S_{n-1}$. Because $\Gal(\hat f/k(t))$ is transitive on the $n$ roots of $\hat f$, we must have $\Gal(\hat f/k(t))=S_n$.

\end{proof}

We are now ready to prove Theorem \ref{T:main}.

\begin{proof}[First proof of Theorem \ref{T:main}]
Again assume that the matroid $M$ is a counterexample with $r_M(E)$ minimal, and among these cases pick one with $\abs{E}$ minimal. Note that the statement is trivially true if $r(M)=1$, thus $r(M)\ge2$ in the minimal counterexample.

Pick $e\in E$. By Lemma \ref{L:sum} $\hat Z_M=\hat Z_{M\setminus e}+v_e\hat Z_{M/e}$. Let $\bar \bv=\bv\setminus \{v_e\}$, and set $k=K(\bar \bv)$. Recall that $\hat Z_M$ is irreducible over $k(v_e)$ by Lemma \ref{L:irr}. We have seen above that $r(M\setminus e)=r(M)=n$ and $r(M/e)=n-1$. As established previously, either $M\setminus e$ or $M/e$ is connected. By assuming a minimal counterexample we have $\Gal(\hat Z_{M\setminus e}/k)=S_n$ or $\Gal(\hat Z_{M/e}/k)=S_{n-1}$. Theorem \ref{T:main} then follows from Lemma \ref{L:U+yV}.
\end{proof}

We will now present an alternative proof of Theorem \ref{T:main}. While it is less efficient than the above proof, it uses a group-theoretical inductive process which is perhaps more intuitive.  We will need to first prove that the theorem holds for circuits.

\begin{Lemma}\label{L:circuit}
Let $C\subseteq E$ be a circuit of a finite matroid $M$.  Then $\Gal(\hat Z_{M|C}/K(\bv_C))=S_{r_M(C)}$.
\end{Lemma}

\begin{proof}
The rank of any proper subset of $C$ is the same as its cardinality, and $r_M(C)=\abs{C}-1$, so:
\[\hat Z_{M|C}(q,\bv)=q^n+\sigma_1 q^{n-1}+\sigma_2 q^{n-2}+\ldots +\sigma_{n-1}q+ (\sigma_n+\sigma_{n+1}),\]
where $\sigma_i$ is the $i$th elementary symmetric polynomial in the $\{v_e\}_{e\in C}$ for each $i$.  The elementary symmetric polynomials are algebraically independent, and thus so too are the coefficients of $\hat Z_{M|C}(q,\bv)$.  It is well known that the Galois group of a polynomial with algebraically independent coefficients is the full symmetric group.
\end{proof}

\begin{proof}[Second proof of Theorem \ref{T:main}]
Let $C$ be a circuit of maximum cardinality in $M$.  By Lemma \ref{L:circuit}, $\Gal(\hat Z_{M|C}/K(\bv_C))=S_{r_M(C)}$.  This will serve as the base case for the induction.

Now, let $A$ be any proper subset of $E$ such that $C\subseteq A$ and $M|A$ is connected, and suppose that $\Gal(\hat Z_{M|A}/K(\bv_A))=S_{r_M(A)}$.  Identify a non-empty independent set $B\subseteq E\setminus A$ of minimal size such that $M|(A\cup B)$ is connected, and let $A'=(A\cup B)$.  We will show that $\Gal(\hat Z_{M|A'}/K(\bv_{A'}))=S_{r_M(A')}$.

By \cite[Lemma 1.3.1]{Oxley:Book} $r_M(A')\leq r_M(A)+r_M(B)$.  By maximality of $C$, any circuit of $M|A'$ has rank at most $r_M(C)$.  By minimality of $B$, any circuit of $M|A'$ not contained in $M|A$ must include at least one element of $A$, so $r_M(B)\leq r_M(C)-1$, and we have $r_M(A')\leq r_M(A)+r_M(C)-1$.

By Lemma \ref{L:restrict}, $S_{r_M(A)}=\Gal(\hat Z_{M|A}/K(\bv_A))\leq \Gal(\hat Z_{M|A'}/K(\bv_{A'}))$.  So $\Gal(\hat Z_{M|A'}/K(\bv_{A'}))$ must contain at least one transposition.  Let $H$ be the group generated by all of the transpositions in $\Gal(\hat Z_{M|A'}/K(\bv_{A'}))$; then $H$ is a direct product of symmetric groups.  As $\Gal(\hat Z_{M|A'}/K(\bv_{A'}))$ is transitive, each of these symmetric groups must have the same degree $i$, which must therefore divide the degree of $\Gal(\hat Z_{M|A'}/K(\bv_{A'}))$. By Lemma \ref{L:irr}, $\hat Z_{M|A'}$ is irreducible, and its Galois group must therefore be transitive of degree $r_M(A')$.  So we have that $ji=r_M(A')$ for some positive integer $j$.

Now, $S_{r_M(A)}$ contains at least one of the transpositions of $H$, so must be a subgroup of one the $S_i$, which means $r_M(A)\leq i$.  So we have: \[jr_M(A)\leq ji=r_M(A')\leq r_M(A)+r_M(C)-1.\]
Suppose that $j\geq 2$. Then $2r_M(A)\leq r_M(A)+r_M(C)-1$, and so $r_M(A)\leq r_M(C)-1$.  This is impossible, as $C\subset A$.  So $j=1$, and hence $i=r_M(A')$.  This means that $H$ is a direct product of symmetric groups of degree $r_M(A')$.  But $H$ is a subgroup of $\Gal(\hat Z_{M|A'}/K(\bv_{A'}))$, which is transitive of degree $r_M(A')$, and so $\Gal(\hat Z_{M|A'}/K(\bv_{A'}))=H=S_{r_M(A')}.$
\end{proof}

Now, in view of the proof of Lemma \ref{L:circuit}, one might wonder if the coefficients of $\hat Z_{M}(q,\bv)$ are algebraically independent for \emph{any} finite connected matroid. This does indeed turn out to be the case, leading us to our third and final proof of Theorem \ref{T:main}.

\begin{proof}[Third proof of Theorem \ref{T:main}]
Let $M$ be a finite connected matroid of rank $r(M)=n\ge1$, and write $\hat Z_{M}(q,\bv)=q^n+a_{n-1}q^{n-1}+\dots+a_1q+a_0\in K[\bv][q]$, where $K$ is an arbitrary field. It suffices to show that the coefficients $a_0,a_1,\dots,a_{n-1}$ are algebraically independent over $K$.

If $n=1$, then $\hat Z_{M}(q,\bv)=q-1+\prod_{e\in E}(v_e+1)$, so the claim clearly holds. Thus we may assume $n\ge2$.

Assume that $M$ is a counterexample in which $\abs{E}$ is minimal. We will use the deletion-contraction identity $\hat Z_M=\hat Z_{M\setminus e}+v_e\hat Z_{M/e}$ of Lemma \ref{L:sum}.  First consider the case that $M\setminus e$ is connected. By the assumption of a minimal counterexample, the coefficients of $\hat Z_{M\setminus e}$ (excluding the leading coefficient $1$) are algebraically independent over $K$. However, these coefficients arise from the coefficients $a_0,a_1,\dots,a_{n-1}$ upon setting $v_e=0$. Of course, an algebraic dependency relation of $a_0,a_1,\dots,a_{n-1}$ over $K$ remains an algebraic dependency relation upon setting $v_e=0$, a contradiction.

Thus $M\setminus e$ is not connected, so we may assume that $M/e$ is connected.  For each $0\leq i\leq n-1$, write $a_i=b_i+v_ec_i$, where $b_i$ and $c_i$ are polynomials in the elements of $\bv_{E\setminus e}$. Each $c_j$ is then the coefficient of $q^j$ in $\hat Z_{M/e}$, so $c_{n-1}=1$ (as $r(M/e)=n-1$) and $c_0,c_1,\dots,c_{n-2}$ are algebraically independent over $K$. As $a_0,a_1,\dots,a_{n-1}$ are algebraically dependent, there is a non-zero polynomial $P$ in $n$ variables over $K$ such that
\[
P(b_0+v_ec_0,\dots,b_{n-2}+v_ec_{n-2},b_{n-1}+v_e)=0.
\]
Let $Q$ be the expansion of $P$ with respect to $v_e$, so that $Q$ is a polynomial in $v_e$ with coefficients in $K[\bv_{E\setminus e}]$.  As the elements of $\bv$ are algebraically independent, these coefficients must be identically zero. Let $d$ be the total degree of $P$.  Then $Q$ has degree $d$ in $v_e$, and the $v_e^d$ term must arise from a $K$-linear sum of products of the form:
\[(b_0+v_ec_0)^{d_0}\dots(b_{n-2}+v_ec_{n-2})^{d_{n-2}}(b_{n-1}+v_e)^{d_{n-1}},\]
where $d_0,\dots,d_{n-1}$ are non-negative integers which sum to $d$.  This means that the coefficient of $v_e^d$ in $Q$ is a $K$-linear combination of monomials of the form $c_0^{d_0}\dots c_{n-2}^{d_{n-2}}$, where $d_i\ge0$ for each $i$, and $d_0+\dots +d_{n-2}\le d$.  The vanishing of this coefficient then implies that the set of such monomials is linearly dependent over $K$, which contradicts our assertion that $c_0,\dots,c_{n-2}$ are algebraically dependent over $K$.
\end{proof}

\begin{Remark}
Sokal showed that the the multivariate Tutte polynomial for matroids factorizes over summands (see \cite[(4.4)]{Sokal:Tutte}).  That is, if $M$ is the direct sum of connected matroids $M_1,M_2$ on the sets $E_1,E_2$ respectively (where $E_1$ and $E_2$ are disjoint and $E=E_1\cup E_2$) then:
\[\hat Z_M(q,\bv)=\hat Z_{M_1}(q,\bv_{E_1})\hat Z_{M_2}(q,\bv_{E_2}).\]
As $\bv_{E_1}$ and $\bv_{E_2}$ are disjoint, there are clearly no algebraic dependencies between the roots of $\hat Z_{M_1}$ and $\hat Z_{M_2}$, so we have that
\[\Gal (\hat Z_M/K(\bv))=\Gal(\hat Z_{M_1}/K(\bv_{E_1}))\times\Gal(\hat Z_{M_2}/K(\bv_{E_2})).\]
Theorem \ref{T:main} then implies that the Galois group of the multivariate Tutte polynomial of any matroid is a direct product of symmetric groups corresponding to the connected direct summands.
\end{Remark}

Finally, we computed the Galois group of the bivariate Tutte polynomial $T_G(x,y)$ over $\mathbb{Q}(y)$ for every biconnected graph $G$ of order $n\leq 10$, and found that all were the symmetric group of degree $n-1$.  As the Tutte polynomial of any connected matroid is irreducible over fields of characteristic zero (as noted in \cite{Merino:Tutte}, this is not necessarily the case for fields of positive characteristic), this would seem to suggest the following:

\begin{Conjecture}\label{C:Tutte}
Let $M$ be a finite connected matroid with positive rank $n=r(M)$, and let $K$ be a field of characteristic zero.  Then the Galois group of the Tutte polynomial $T_M(x,y)$ over $K(y)$ is the symmetric group of degree $n$.
\end{Conjecture}

As remarked previously, the bivariate Tutte polynomial is essentially a specialization of the multivariate version.  This means that Theorem \ref{T:main} would follow from a proof of Conjecture \ref{C:Tutte} for fields of characteristic zero.

Interestingly, specializing the Tutte polynomial further produces a range of different Galois groups.  For example, it was shown in \cite{pjc11} that all of the transitive permutation groups of degree at most 5 apart from $C_5$ appear as Galois groups of just one family of chromatic polynomials. Furthermore, Morgan \cite{Morgan:Thesis} showed that a range of transitive groups of higher degree occur for chromatic polynomials of graphs on up to 10 vertices.

\bibliographystyle{plain}
\bibliography{mybib}

\end{document}